\documentclass{article}
\usepackage[cp1251]{inputenc} 
\usepackage[tbtags]{amsmath}
\usepackage{amsfonts,amssymb,mathrsfs,amscd,comment}
\usepackage{amscd,amsthm,amsmath,amssymb}

\newcommand{\p}{\mathbb{P}}

\newcommand{\com}{\mathbb{C}}

\newcommand{\map}{\longrightarrow}

\newtheorem{theorem}[equation]{Theorem}

\newtheorem{lemma}[equation]{Lemma}

\theoremstyle{remark}
\newtheorem{remark}[equation]{Remark}

\newtheorem{ex}[equation]{Example}

\begin{document}

\title{\bf\large\MakeUppercase{Maximum likelihood degree of surjective rational maps}}

\author{Ilya Karzhemanov~\footnote{Laboratory of AGHA, Moscow Institute of Physics and Technology, 9 Institutskiy per., Dolgoprudny,
Moscow Region, 141701, Russia / karzhemanov.iv@mipt.ru.}}
\date{}

\maketitle

\begin{abstract}
With any \emph{surjective rational map} $f: \p^n \dashrightarrow
\p^n$ of the projective space we associate a numerical invariant
(\emph{ML degree}) and compute it in terms of a naturally defined
vector bundle $E_f \map \p^n$.
\end{abstract}

\makeatletter
\renewcommand{\@makefnmark}{}
\makeatother

\footnotetext{{\it MS 2010 classification}: 14E05, 14N10, 14F10.}

\bigskip

With a collection of effective divisors $D_0,\ldots,D_m$ in the
projective space $\p := \com\p^n$ is associated the \emph{maximum
likelihood degree} $(-1)^ne_{\text{top}}(\p\setminus D)$, $D :=
\displaystyle\bigcup_{i=0}^m D_i$. Alternatively, letting
$\Omega^1_{\p}(\log D)$ be the \emph{Saito's sheaf}, i.e. the
\emph{double dual} of the sheaf of logarithmic differential
$1$-forms, one computes the ML degree as the top Chern class
$c_n(\Omega^1_{\p}(\log D))$ (we refer to \cite{catanese-et-al},
\cite{huh-1} for basic properties of ML degree, its connections
with algebraic statistics, topology of arrangements,
combinatorics, etc.). Note however that it is difficult to compute
$c_n(\Omega^1_{\p}(\log D))$ in general (when $D$ is not SNC).

In the present note, we study the ML degree under the condition
that defining polynomials $f_i$ of $D_i$, $0 \le i \le m = n$,
span the linear system of a \emph{surjective rational map} $f: \p
\dashrightarrow \p$ (see \cite{aru} and \cite{ilya-ilya} for some
aspects of such maps). Our main result (proved along the lines
that follow) is the next

\begin{theorem}
\label{theorem:main} In the previous setting, the ML degree
$c_n(\Omega^1_{\p}(\log D))$ is equal to the coefficient of $z^n$
in $\frac{(1 - z\mathcal{O}_{\p}(1))^{n+1}}{\prod_{i=0}^n(1 -
z\mathcal{O}_{\p}(D'_i))}$, where $\bigcup_{i=0}^n D'_i =:
D_{\text{red}}$ is the reduced scheme associated with $D$ (so that
$D = D_{\text{red}}$ as sets).
\end{theorem}

For a vector bundle $E$ over $\p$, given by an affine open cover
$\p = \displaystyle\cup_{\alpha}\,U_{\alpha}$ and transition
functions $g_{\alpha\beta}: U_{\alpha} \cap U_{\beta} \map
\text{GL}(r,\com)$, the \emph{pullback} $f^*(E)$ on
$\p\setminus{\{\Sigma :=\ \text{base locus of}\, f\}}$ is defined
as usual (due to the surjectivity of $f$), via
$f^{-1}(U_{\alpha})$ and $f^*(g_{\alpha\beta})$. Note that all
$f^{-1}(U_{\alpha})$ are affine open in $\p$. Let
$\displaystyle\cup_{k}\,U_{\alpha,k}$ be an affine open cover of
$f^{-1}(U_{\alpha})$ such that $\p =
\displaystyle\cup_{\alpha,k}\,U_{\alpha,k}$. Then, since
$\text{codim}\,\Sigma > 1$ and $f^*(g_{\alpha\beta})$ are
\emph{algebraic}, every $f^*(g_{\alpha\beta})$ extends through
$U_{\alpha,k} \cap U_{\beta,m} \cap \Sigma$ to each $U_{\alpha,k}
\cap U_{\beta,m}$. Furthermore, the $1$-cocyle property of
$f^*(g_{\alpha\beta})$ (considered on
$(\displaystyle\cup_{k}\,U_{\alpha,k}) \cap
(\displaystyle\cup_{m}\,U_{\beta,m}) \supseteq f^{-1}(U_{\alpha})
\cap f^{-1}(U_{\beta})$) is preserved and one gets a vector
bundle, over $\p$, which we denote again by $f^*(E)$.

Further, let $x_0,\ldots,x_n$ be projective coordinates on $\p$
such that $f^*(x_i) = f_i$. Denote by $H$ the union of coordinate
hyperplanes $H_i := (x_i = 0) \subset \p$. There is an exact
sequence
\begin{equation}
\label{eq-1} 0 \map \Omega^1_{\p} \stackrel{\psi_H}{\map}
\Omega^1_{\p}(\log H) \stackrel{\varphi_H}{\map} \bigoplus_{i =
0}^n \mathcal{O}_{H_i} \map 0
\end{equation}
(see e.g. \cite[Lemma 2]{catanese-et-al}). We have
$f^*(\mathcal{O}_{\p}) = \mathcal{O}_{\p}$ and
$f^*(\mathcal{O}_{\p}(H)) = \mathcal{O}_{\p}(D)$ by construction.
Then \eqref{eq-1} pulls back to an exact sequence
\begin{equation}
\label{eq-1.5} f^*(\Omega^1_{\p}) \stackrel{\psi_D}{\map}
f^*(\Omega^1_{\p}(\log H)) \stackrel{\varphi_D}{\map}
f^*(\bigoplus_{i = 0}^n \mathcal{O}_{H_i}) = \bigoplus_{i = 0}^n
\mathcal{O}_{D_i}.
\end{equation}
Note however that the morphism $\psi_D := f^*(\psi_H)$ (resp.
$\varphi_D := f^*(\varphi_H)$) need not be injective (resp.
surjective) --- see below.

\begin{lemma}
\label{f-star-o-log} $f^*(\Omega^1_{\p}(\log H)) =
\Omega^1_{\p}(\log D)$.
\end{lemma}

\begin{proof}[Proof]
The bundle $\Omega^1_{\p}(\log H)$ (resp. $\Omega^1_{\p}(\log D)$)
is trivial over an affine open set not containing $H$ (resp. $D$).
Hence, as $f^*(\mathcal{O}_{\p}) = \mathcal{O}_{\p}$, it suffices
to restrict to an affine open $U \subset \p^n$ (resp. $f^{-1}(U)$)
such that $U \cap H \ne \emptyset$ (we may also assume that $x_0
\ne 0$ on $U$). Then $\Omega^1_{\p}(\log H)\big\vert_U$ is
generated by the local sections $\displaystyle\sum_{i = 1}^n
c_i\log x_i$, $c_i \in \com$, whereas $f^*(\Omega^1_{\p}(\log
H))\big\vert_{f^{-1}(U)}$ is generated by $\displaystyle\sum_{i =
1}^n c_i\log f^*(x_i)$ (as usual we take double duals when
needed). This yields $f^*(\Omega^1_{\p}(\log
H))\big\vert_{f^{-1}(U)} = \Omega^1_{\p}(\log
D)\big\vert_{f^{-1}(U)}$ and the result follows.
\end{proof}

Before finding $f^*(\Omega^1_{\p})$ we need an auxiliary
construction. Namely, put $d_f := \deg f_i$ and consider the
subspace $V \subset H^0(\p,\mathcal{O}_{\p}(d_f))$ spanned by
$f_0,\ldots,f_n$. Recall that by Kodaira's construction of
rational maps via linear systems, every point $p \in
f(\p\setminus{\Sigma})$ is represented by hyperplane $H_p \subset
V$, which consists of all polynomials from $V$ vanishing at
$f^{-1}(p)$. Then, since $f$ is surjective, this defines a
\emph{vector bundle} $E_f \map \p = f(\p\setminus{\Sigma})$, with
fibers $E_{f,p} = H_p$ for all $p$, and an exact sequence
\begin{equation}
\label{eq-2} 0 \map \mathcal{L} \map \com^{n+1} \map E_f \map 0
\end{equation}
for some line bundle $\mathcal{L}$. It is easy to prove (by
induction on $n$) that $\mathcal{L} = \mathcal{O}_{\p}(-n-1)$.
This implies that both $E_f$ and $f^*(E_f)$ are generated by
global sections.

We now prove the following (``Hurwitz-type''):

\begin{lemma}
\label{f-star-o-1} $f^*(\Omega^1_{\p}) \subseteq
\Omega^1_{\p}\otimes_{\mathcal{O}_{\p}}\mathcal{O}_{\p}(-d_f +
1)$.
\end{lemma}

\begin{proof}
Each global section of $f^*(E_f)$ is given by some choice of a
basis ($=\{x_0,\ldots,x_n\}$) in $\com^{n+1}$ and a way every $p
\in \p$ (identified with $\sum p_i x_i$ for $p_i \in \com$) is
represented by a point in $V \simeq \com^{n+1} = H^0(\p,E_f)$.
This yields a \emph{surjection}
$$
\mathcal{H}om_{\mathcal{O}_{\p}}(\mathcal{O}_{\p}(1), E_f
\otimes_{\mathcal{O}_{\p}} \mathcal{O}_{\p}(d_f)) = E_f
\otimes_{\mathcal{O}_{\p}} \mathcal{O}_{\p}(d_f - 1)
\twoheadrightarrow f^*(E_f)
$$
of vector bundles generated by global sections.

Now observe that $E_f \simeq T_{\p}\ (= \text{the dual of}\
\Omega^1_{\p})$ by \eqref{eq-2} and \cite[Theorem 3.1]{hwang}.
Hence $f^*(\Omega^1_{\p})$ embeds into
$\Omega^1_{\p}\otimes_{\mathcal{O}_{\p}}\mathcal{O}_{\p}(-d_f +
1)$ by duality.
\end{proof}

Note that $\Omega^1_{\p}(\log D) = \Omega^1_{\p}(\log
D_{\text{red}})$ (cf. the proof of Lemma~\ref{f-star-o-log}).
Hence $\varphi_D(\Omega^1_{\p}(\log D)) = \bigoplus_{i = 0}^n
\mathcal{O}_{D'_i}$. Further, it follows from \eqref{eq-1.5} and
Lemma~\ref{f-star-o-1} that the kernel of $\varphi_D$ is a
subsheaf of $\psi_D(\Omega^1_{\p} \otimes \mathcal{O}_{\p}(-d_f +
1))$, whose general local section is easily seen (by restricting
on $\p\setminus\Sigma$) to coincide with a holomorphic $1$-form,
which vanishes \emph{at most} on $D_{\text{red}}$. One actually
finds that this is a \emph{subbundle} of $\Omega^1_{\p}$ generated
by all such $1$-forms. Thus we get $\text{Ker}\,\varphi_D =
\Omega^1_{\p}$ and an exact sequence
$$
0 \map \Omega^1_{\p} \map \Omega^1_{\p}(\log D) \map \bigoplus_{i
= 0}^n \mathcal{O}_{D'_i} \map 0.
$$
Taking the total Chern class of the latter concludes the proof of
Theorem~\ref{theorem:main}.

\begin{remark}
\label{remark:moduli-etc} We summarize that any $f$ defines,
\emph{canonically}, a fiberwise non-degenerate element $e \in
\text{Hom}\,(\p; T_{\p}, T_{\p} \otimes \mathcal{O}_{\p}(d_f -
1))$. This can also be seen as follows. Namely, the embedding
$\mathcal{L} \subset \com^{n+1}$ in \eqref{eq-2} is given by some
global sections $s_0,\ldots,s_n \in H^0(\p,\mathcal{O}_{\p}(n +
1))$, so that $x_i \mapsto s_i$, $0 \le i \le n$, defines a
\emph{regular surjective} self-map of $\p$. This yields a family
(a ``field'') $\{H_p\}$ of hyperlines on $\p \ni p$. After
choosing $e$, one gets another family $\{H'_p\}$, where $H'_p
\simeq H_p$ are spaces of forms of degree $d_f$ and
$\displaystyle\bigcup H'_p = V$. Identify $H'_p$ with the set of
corresponding hypersurfaces that vanish at $p$. The map $f$ is now
obtained by sending each $p \in H'_p$ to $H_p$ (it is defined
exactly on $\p\setminus\displaystyle\bigcap H'_p$). One thus
obtains a description of the moduli spaces of surjective maps $f$.
It would be interesting to relate this picture with \cite{don},
where the moduli of degree $k$ rational self-maps of $\p^1$ were
interpreted as the moduli of (pairs of) \emph{monopoles}, having
magnetic charge $k$.
\end{remark}

\begin{ex}
\label{ex:comp} The need for $D_{\text{red}}$ in
Theorem~\ref{theorem:main} is justified by the \emph{Frobenius
map} $f$, given by $f_i := x_i^{d_f}$, $0 \le i \le n$; ML degree
of $f$ equals $(-1)^ne_{\text{top}}((\com^*)^n) = {\bf 0}$ in this
case. Further, one computes the ML degree of $f$ in \cite[Example
1.6]{ilya-ilya} to be ${\bf 9}$, which can also be seen directly
from \cite[Corollary 6]{catanese-et-al} (here the divisors $D_i$
satisfy the \emph{GNC condition}). Indeed, in this case $D_i$ are
reduced and $\deg f_i = 2$ for all $i$, so that the expression
with Chern classes from Theorem~\ref{theorem:main} becomes
$$
\frac{(1 - z\mathcal{O}_{\p}(1))^3}{(1 - z\mathcal{O}_{\p}(2))^3}
= (1 - z\mathcal{O}_{\p}(1))^3(1 + z\mathcal{O}_{\p}(2) + 4z^2)^3
=
$$
$$
= (1 + z\mathcal{O}_{\p}(1) + 2z^2)^3 = 1 + 3(z\mathcal{O}_{\p}(1)
+ 2z^2) + 3(z\mathcal{O}_{\p}(1) + 2z^2)^2 = 1 +
z\mathcal{O}_{\p}(3) + {\bf 9}z^2.
$$
\end{ex}

\bigskip

\thanks{{\bf Acknowledgments.}
I am grateful to anonymous referee for helpful comments and
corrections.

\end{document}